\documentclass[11pt,a4paper,leqno]{amsart}
\usepackage[T1]{fontenc}
\usepackage{lmodern}
\usepackage{amsmath,amssymb,mathtools}
\usepackage[margin=25mm]{geometry}
\usepackage{microtype}
\usepackage{booktabs,array}
\usepackage[hidelinks]{hyperref}
\hypersetup{pdftitle={Distinctness and algebraic degrees of inverted Kloosterman sums},pdfauthor={Danyil Kyselov},pdfsubject={Exact distinctness over finite field extensions}}
\numberwithin{equation}{section}
\newtheorem{theorem}{Theorem}[section]
\newtheorem{proposition}[theorem]{Proposition}
\newtheorem{lemma}[theorem]{Lemma}
\newtheorem{corollary}[theorem]{Corollary}
\theoremstyle{remark}
\newtheorem{remark}[theorem]{Remark}
\newcommand{\F}{\mathbb F}
\newcommand{\Q}{\mathbb Q}

\newcommand{\Tr}{\operatorname{Tr}}
\newcommand{\Gal}{\operatorname{Gal}}
\newcommand{\IK}{\operatorname{IK}}

\newcommand{\cO}{\mathcal O}
\newcommand{\red}[1]{\overline{#1}}
\newcommand{\degree}{\operatorname{deg}}

\allowdisplaybreaks[1]
\title[Distinctness of inverted Kloosterman sums]{Distinctness and algebraic degrees of inverted Kloosterman sums}
\author{Danyil Kyselov}
\address{Odessa I. I. Mechnikov National University, 2 Dvoryanskaya Street, 65086 Odessa, Ukraine}
\email{daniil.kiselev@stud.onu.edu.ua}
\date{October 15, 2026}
\begin{document}
\begin{abstract}
We study the distinctness and algebraic degrees of inverted Kloosterman sums over finite extensions of $\F_p$. For each $d\ge2$, we prove that the $(d-1)$-dimensional sums over $\F_{p^2}$ are equal precisely when their parameters are Frobenius conjugate, provided that $p\equiv1\pmod d$, $p>2d+1$, and $p$ does not divide the nonzero integer $2(d+1)^d-\binom{2d}{d}^{2}$. For two-dimensional sums over $\F_{p^r}$, the same conclusion holds for $2\le r\le5$, $p\equiv1\pmod3$, and $p>3r+1$. In both cases we determine the algebraic degree of every sum in terms of the minimal polynomial of its parameter. The proof combines a cyclic-carry analysis of a Gauss-sum expansion with the recovery of power traces from a formal logarithm.
\end{abstract}
\maketitle

\section{Introduction}
Let $p$ be an odd prime, let $q=p^r$ with $r\ge1$, and put
\[
 \psi(x)=\zeta_p^{\Tr_{\F_q/\F_p}(x)},\qquad \zeta_p=e^{2\pi i/p}.
\]
For $d\ge2$ and $b\in\F_q^\times$, the $(d-1)$-dimensional inverted Kloosterman sum is
\begin{equation}\label{eq:def}
 \IK_{d-1}(q,b)=
 \sum_{\substack{x_1,\ldots,x_d\in\F_q^\times\\
 x_1\cdots x_d=b\\x_1+\cdots+x_d\ne0}}
 \psi\!\left(\frac1{x_1+\cdots+x_d}\right).
\end{equation}
This is an algebraic integer in $\Q(\zeta_p)$. We write
\[
 \degree\IK_{d-1}(q,b)=[\Q(\IK_{d-1}(q,b)):\Q].
\]
Lin and Wan studied the complex and $p$-adic properties of these sums in~\cite{LWfields}. They subsequently proved~\cite{LWdegree} that
\begin{equation}\label{eq:LWdegree}
 \degree\IK_{d-1}(p,b)=\frac{p-1}{\gcd(d,p-1)}
 \qquad(b\in\F_p^\times)
\end{equation}
and asked for an analogue over proper extensions of $\F_p$. We obtain degree formulas over quadratic extensions in arbitrary dimension and, in dimension two, over extensions of degrees at most five. The underlying result is a distinctness theorem: under the hypotheses stated below, two sums are equal if and only if their parameters belong to the same Frobenius orbit.

The relation between distinctness and algebraic degree follows from two symmetries. For $a\in\F_p^\times$, denote by $\sigma_a\in\Gal(\Q(\zeta_p)/\Q)$ the automorphism satisfying $\sigma_a(\zeta_p)=\zeta_p^a$. Then
\begin{align}
 \IK_{d-1}(q,b^p)&=\IK_{d-1}(q,b),\label{eq:frob}\\
 \sigma_a\IK_{d-1}(q,b)&=\IK_{d-1}(q,ba^{-d}).\label{eq:galois}
\end{align}
The first identity follows by applying Frobenius to the variables in~\eqref{eq:def}; the second follows by substituting $x_i=ay_i$. Thus the parameter map factors through Frobenius orbits, whereas its Galois conjugates are obtained by scalar multiplication of the parameter.

For $d\ge2$, set
\begin{equation}\label{eq:D}
 D_d=2(d+1)^d-\binom{2d}{d}^{\!2}.
\end{equation}
The integer $D_d$ is nonzero for all $d\ge2$; see Lemma~\ref{lem:Dnonzero}.

\begin{theorem}[Quadratic extensions]\label{thm:quadratic}
Let $d\ge2$, and suppose that
\begin{equation}\label{eq:quadcond}
 p\equiv1\pmod d,\qquad p>2d+1,\qquad p\nmid D_d.
\end{equation}
For $b,c\in\F_{p^2}^\times$, we have
\begin{equation}\label{eq:quaddist}
 \IK_{d-1}(p^2,b)=\IK_{d-1}(p^2,c)
 \quad\Longleftrightarrow\quad c\in\{b,b^p\}.
\end{equation}
Moreover,
\begin{equation}\label{eq:quaddegree}
 \degree\IK_{d-1}(p^2,b)=
 \begin{cases}
 (p-1)/d,&\Tr_{\F_{p^2}/\F_p}(b)\ne0,\\[2pt]
 (p-1)/\gcd(2d,p-1),&\Tr_{\F_{p^2}/\F_p}(b)=0.
 \end{cases}
\end{equation}
\end{theorem}
In particular, for fixed $d$, the conclusion holds for all sufficiently large primes $p\equiv1\pmod d$.

For the degree formula in the second theorem, let
\[
 P_b(X)=X^s+c_1X^{s-1}+\cdots+c_s\in\F_p[X]
\]
be the minimal polynomial of $b\ne0$ over $\F_p$, and define
\begin{equation}\label{eq:gb}
 g_b=\gcd\{k:1\le k\le s,\ c_k\ne0\}.
\end{equation}
Since $c_s\ne0$, the integer $g_b$ is well defined.

\begin{theorem}[Two-dimensional sums]\label{thm:twodim}
Suppose that
\begin{equation}\label{eq:twocond}
 2\le r\le5,\qquad p\equiv1\pmod3,\qquad p>3r+1.
\end{equation}
For $b,c\in\F_{p^r}^\times$, we have
\begin{equation}\label{eq:twodist}
 \IK_2(p^r,b)=\IK_2(p^r,c)
 \quad\Longleftrightarrow\quad
 c=b^{p^j}\text{ for some }0\le j<r.
\end{equation}
For every $b\in\F_{p^r}^\times$,
\begin{equation}\label{eq:twodegree}
 \degree\IK_2(p^r,b)=\frac{p-1}{\gcd(3g_b,p-1)}.
\end{equation}
\end{theorem}
The degree formula includes parameters in proper subfields of $\F_{p^r}$; no assumption on $[\F_p(b):\F_p]$ is required.

The method is related to the $p$-adic approach to distinctness developed by Wan for classical hyper-Kloosterman sums~\cite{Wan}. For that family, Kononen, Rinta-aho and V\"a\"an\"anen~\cite{KRV} determined the generated fields under a distinctness hypothesis. Related questions for other exponential sums and generalized Kloosterman sums were considered by Lin~\cite{Lin25,Lin26}.

Our starting point is the Gauss-sum expansion of Lin and Wan~\cite{LWdegree}. Over an extension field, several characters can contribute at the same $p$-adic order. When $d\mid p-1$, the terms below a fixed precision bound are exactly those whose digit expansions have no cyclic carries. Their leading coefficients assemble into a product of one-variable formal series evaluated at the Frobenius conjugates of $b^{-1}$. Taking the logarithm separates the power traces of these conjugates. Newton's identities, or a Vandermonde argument using consecutive power traces, then give a general distinctness criterion.

Sections~\ref{sec:gauss} and~\ref{sec:carries} establish the Gauss-sum expansion and the required valuation bounds. The trace-recovery argument and the distinctness criterion are proved in Section~\ref{sec:traces}. Section~\ref{sec:degree} relates the Galois stabilizer to scalar symmetries of the minimal polynomial. The main theorems are proved in Section~\ref{sec:main}, and two consequences are given in Section~\ref{sec:consequences}.

\subsection*{Notation}
Unless otherwise indicated, $\Tr$ denotes the absolute trace to $\F_p$, and an overline denotes reduction modulo the maximal ideal of the relevant valuation ring. We use the rising factorial $(u)_0=1$, $(u)_j=u(u+1)\cdots(u+j-1)$ for $j\ge1$, and write $[T^j]f(T)$ for the coefficient of $T^j$ in $f(T)$. Rational coefficients are reduced modulo $p$ only when their denominators are prime to $p$.

\section{Gauss sums and \texorpdfstring{$p$}{p}-adic preliminaries}\label{sec:gauss}
For a multiplicative character $\chi$ of $\F_q^\times$, put
\[
 G(\chi)=\sum_{x\in\F_q^\times}\chi(x)\psi(x).
\]
Thus $G(1)=-1$. We recall the following expansion from Lin and Wan~\cite[\S2, formula (2.4)]{LWdegree}, including its proof to fix the normalizations.

\begin{lemma}\label{lem:gauss}
For $b\in\F_q^\times$,
\begin{equation}\label{eq:gauss}
 q(q-1)\IK_{d-1}(q,b)=-(q-1)^d+(-1)^d
 +\sum_{\chi\ne1}\chi^d(-1)\chi^{-1}(b)G(\chi^{-d})^2G(\chi)^d.
\end{equation}
\end{lemma}
\begin{proof}
By additive and multiplicative character orthogonality, the left-hand side of~\eqref{eq:gauss} equals
\[
 \sum_\chi\chi^{-1}(b)\sum_{u\in\F_q}\sum_{\lambda\in\F_q^\times}
 \psi(-u\lambda+\lambda^{-1})
 \prod_{i=1}^d\sum_{x_i\in\F_q^\times}\chi(x_i)\psi(ux_i).
\]
The terms with $u=0$ contribute $-(q-1)^d$. For $u\ne0$, each inner sum is $\chi^{-1}(u)G(\chi)$. Summing first over $u$ and then putting $t=\lambda^{-1}$ gives
\[
 \sum_{\lambda\ne0}\psi(\lambda^{-1})
 \sum_{u\ne0}\chi^{-d}(u)\psi(-u\lambda)
 =\chi^d(-1)G(\chi^{-d})^2.
\]
The contribution of the trivial character is $G(1)^{d+2}=(-1)^d$, which gives~\eqref{eq:gauss}.
\end{proof}

Fix an embedding $\Q(\zeta_p)\hookrightarrow\overline{\Q}_p$. Let $K/\Q_p$ be the unramified extension of degree $r$, identified with a subfield of $\overline{\Q}_p$, and let $\omega:\F_q^\times\to\cO_K^\times$ be the Teichm\"uller character. Choose $\pi\in\Q_p(\zeta_p)$ such that
\begin{equation}\label{eq:pi}
 \pi^{p-1}=-p,\qquad \pi\equiv\zeta_p-1\pmod{\pi^2},
\end{equation}
and normalize $v_\pi$ by $v_\pi(\pi)=1$. If
\[
 m=m_0+m_1p+\cdots+m_{r-1}p^{r-1},\qquad 0\le m_i\le p-1,
\]
write $s(m)=\sum_i m_i$.

\begin{lemma}\label{lem:stick}
For $0<m<q-1$, there is a unit $U_m\in\cO_K^\times$ such that
\begin{equation}\label{eq:stick}
 G(\omega^{-m})=\pi^{s(m)}U_m,
 \qquad \red{U_m}=-\frac1{m_0!\cdots m_{r-1}!}.
\end{equation}
Equivalently, $U_m\equiv-(m_0!\cdots m_{r-1}!)^{-1}\pmod{p\cO_K}$.
\end{lemma}
\begin{proof}
The valuation and residue in~\eqref{eq:stick} are given by the Stickelberger congruence, or by the Gross--Koblitz formula~\cite{GK}. It remains to show that the normalized Gauss sum belongs to the unramified field $K$.

Extend $\sigma_a$ to $K(\zeta_p)$ by fixing $K$. Then
\[
 \sigma_aG(\omega^{-m})=\omega(a)^mG(\omega^{-m}),
 \qquad \sigma_a(\pi)=\omega(a)\pi.
\]
Indeed, $\sigma_a(\pi)/\pi$ is a $(p-1)$st root of unity with residue $a$, by~\eqref{eq:pi}. Since $s(m)\equiv m\pmod{p-1}$, the quotient $G(\omega^{-m})/\pi^{s(m)}$ is fixed by $\Gal(K(\zeta_p)/K)$, and therefore lies in $K$. Its valuation is zero, so it belongs to $\cO_K^\times$. The final congruence follows because $p\cO_K$ is the maximal ideal of $\cO_K$.
\end{proof}

\section{Cyclic carries and valuation bounds}\label{sec:carries}
Assume throughout this section that $d\mid p-1$. Put
\begin{equation}\label{eq:parameters}
 N=p^r-1,\qquad h=\frac{p-1}{d},\qquad
 M=\frac Nd=h(1+p+\cdots+p^{r-1}),\qquad B=r(p-1).
\end{equation}
For $0<m<N$ with $N\nmid dm$, let $t$ be the representative of $dm$ modulo $N$ in $\{1,\ldots,N-1\}$. By Lemma~\ref{lem:stick}, the valuation of the Gauss product corresponding to $\chi=\omega^{-m}$ in~\eqref{eq:gauss} is
\begin{equation}\label{eq:W}
 W(m):=v_\pi\bigl(G(\omega^{dm})^2G(\omega^{-m})^d\bigr)
 =2s(N-t)+d\,s(m).
\end{equation}

\begin{lemma}[Cyclic carries]\label{lem:carries}
Write $m=\sum_i m_ip^i$ and $t=\sum_i t_ip^i$. There are integers $c_i\in\{0,\ldots,d-1\}$, with $c_{-1}=c_{r-1}$, such that
\begin{equation}\label{eq:carry}
 dm_i+c_{i-1}=t_i+pc_i\qquad(0\le i<r).
\end{equation}
If $C=\sum_i c_i$, then
\begin{equation}\label{eq:carryW}
 s(t)=d\,s(m)-(p-1)C,\qquad
 W(m)=2B-s(t)+(p-1)C.
\end{equation}
If $C\ge1$, then $W(m)\ge B+p$. If $C=0$, put $j_i=h-m_i$ and $k=\sum_i j_i$. Then $j_i\ge0$ and
\begin{equation}\label{eq:nocarryW}
 W(m)=B+dk.
\end{equation}
Consequently, the terms with $N\nmid dm$ and $W(m)<B+p-1$ are exactly those for which
\begin{equation}\label{eq:lowm}
 m=M-\sum_{i=0}^{r-1}j_ip^i,
 \qquad j_i\ge0,\qquad 1\le\sum_i j_i<h.
\end{equation}
\end{lemma}
\begin{proof}
Put $c_{r-1}=\lfloor dm/N\rfloor$. The identity
\[
 dm+c_{r-1}=t+c_{r-1}p^r
\]
allows ordinary base-$p$ multiplication to be carried out with initial carry $c_{-1}=c_{r-1}$. At each step the carry lies between $0$ and $d-1$, and the final carry is again $c_{r-1}$. This gives \eqref{eq:carry}.

Summing \eqref{eq:carry} over the digit positions yields the first equality in \eqref{eq:carryW}. Since $s(N-t)=B-s(t)$, equation \eqref{eq:W} gives the second. Now $0<t<N$ implies $s(t)\le B-1$, so $C\ge1$ gives
\[
 W(m)\ge2B-(B-1)+(p-1)=B+p.
\]
If $C=0$, all carries vanish and $dm_i=t_i\le p-1$. Hence $m_i\le h$. Substitution of $m_i=h-j_i$ in \eqref{eq:W} proves \eqref{eq:nocarryW}. The case $k=0$ would give $m=M$ and $N\mid dm$, which has been excluded. Thus $W(m)<B+p-1$ is equivalent to $1\le k<h$.

Conversely, for every weak composition $k=\sum_i j_i$ with $1\le k<h$, the digits $m_i=h-j_i$ lie in $\{1,\ldots,h\}$, are not all equal to $h$, and produce no carries. Their integer $m$ satisfies $0<m<M$ and hence $N\nmid dm$. This proves the last assertion.
\end{proof}

The characters satisfying $\chi^d=1$ are treated separately.
\begin{lemma}\label{lem:exceptionalchars}
If $0<m<N$ and $N\mid dm$, then $m=\ell M$ for a unique $1\le\ell\le d-1$. Its base-$p$ digits are all $\ell h$, and
\begin{equation}\label{eq:exceptionalW}
 v_\pi\bigl(G(\omega^{dm})^2G(\omega^{-m})^d\bigr)=\ell B.
\end{equation}
\end{lemma}
\begin{proof}
Since $N=dM$, the condition $N\mid dm$ is equivalent to $M\mid m$. For $m=\ell M$, all digits are $\ell h<p$, while $G(\omega^{dm})=G(1)=-1$. Hence Lemma~\ref{lem:stick} gives $v_\pi(G(\omega^{-m})^d)=dr\ell h=\ell B$.
\end{proof}

\section{Power traces and distinctness}\label{sec:traces}
Consider the formal power series over $\Q$
\begin{equation}\label{eq:fseries}
 f_d(T)=\sum_{j\ge0}a_j(d)T^j,\qquad
 a_j(d)=\frac{(1/d)_j^d}{((dj)!)^2},\qquad
 \log f_d(T)=\sum_{j\ge1}\lambda_j(d)T^j.
\end{equation}
Retain the notation $h=(p-1)/d$ from~\eqref{eq:parameters}. For $j<h$, the coefficients $a_j(d)$ and $\lambda_j(d)$ are $p$-integral, since $dj<p-1$ and $j<p$. Thus, for $0\le k<h$ and $b\in\F_q^\times$, we may define
\begin{equation}\label{eq:H}
 H_k(b)=[T^k]\prod_{i=0}^{r-1}f_d(b^{-p^i}T)\in\F_p.
\end{equation}
Here the product is formed modulo $T^h$ after reducing its coefficients modulo $p$. Frobenius permutes the factors, so each $H_k(b)$ belongs to $\F_p$.

\begin{proposition}[Recovery of power traces]\label{prop:traces}
Assume $d\mid p-1$. If
\[
 \IK_{d-1}(q,b)=\IK_{d-1}(q,c),
\]
then $b^M=c^M$, and, for every $1\le k<h$,
\begin{align}
 H_k(b)&=H_k(c),\label{eq:Hequal}\\
 \lambda_k(d)\bigl(\Tr(b^{-k})-\Tr(c^{-k})\bigr)&=0
 \quad\text{in }\F_p.\label{eq:traceequal}
\end{align}
\end{proposition}
\begin{proof}
Apply~\eqref{eq:gauss} to the difference of the two sums. Lemmas~\ref{lem:carries} and~\ref{lem:exceptionalchars} show that $m=M$ is the unique term of valuation $B$, all other terms having larger valuation. Dividing by $\pi^B$ and reducing gives
\[
 \frac{(-1)^d}{(h!)^{dr}}(b^M-c^M).
\]
since $\chi^d(-1)=1$ for $m=M$. It follows that $b^M=c^M$, and hence their Teichm\"uller lifts satisfy
\begin{equation}\label{eq:A}
 \omega(b)^M=\omega(c)^M=:A.
\end{equation}
Consequently, all terms with $m=\ell M$ cancel in the difference before reduction.

For a term of the form~\eqref{eq:lowm}, write $k=\sum_i j_i<h$. The digits of $N-dm$ are $dj_i$, so Lemma~\ref{lem:stick} gives the following residue after division of the Gauss product by $\pi^{B+dk}$:
\[
 (-1)^d\prod_{i=0}^{r-1}\frac1{(h-j_i)!^d(dj_i)!^2}.
\]
Its parameter factor is
\[
 \omega(b)^m=A\prod_i\omega(b)^{-j_ip^i},
\]
and its sign $\chi^d(-1)$ is $(-1)^{dk}$. Since $h\equiv-1/d\pmod p$, we have
\begin{equation}\label{eq:factorial}
 \left(\frac{h!}{(h-j)!}\right)^d
 \equiv(-1)^{dj}(1/d)_j^d\pmod p.
\end{equation}
The signs in~\eqref{eq:factorial} cancel the character factor $(-1)^{dk}$. Summation over the weak compositions $j_0+\cdots+j_{r-1}=k$ therefore gives $H_k(b)$, up to a common nonzero factor.

More precisely, put
\[
 \Delta=qN\bigl(\IK_{d-1}(q,b)-\IK_{d-1}(q,c)\bigr).
\]
After cancellation of the terms $m=\ell M$, we obtain
\begin{equation}\label{eq:precision}
 \frac{\Delta}{\pi^B}=\sum_{k=1}^{h-1}\pi^{dk}C_k+R,
 \qquad C_k\in\cO_K,\qquad v_\pi(R)\ge p-1,
\end{equation}
where
\begin{equation}\label{eq:Cres}
 \red{C_k}=\frac{(-1)^d\red A}{(h!)^{dr}}
 \bigl(H_k(b)-H_k(c)\bigr).
\end{equation}
The assertion $C_k\in\cO_K$ follows from Lemma~\ref{lem:stick}. By Lemma~\ref{lem:carries}, terms with nonzero carries have valuation at least $B+p$, and the remaining omitted terms have valuation at least $B+p-1$. After division by $\pi^B$, they contribute to the remainder $R$ in~\eqref{eq:precision}.

If any $C_k$ had nonzero residue, choose the least such $k$. For $j<k$, we would have $C_j\in p\cO_K$, and hence
\[
 v_\pi(\pi^{dj}C_j)\ge dj+p-1>dk\qquad(j<k).
\]
For $j>k$, the term $\pi^{dj}C_j$ also has valuation greater than $dk$, while $v_\pi(R)\ge p-1>dk$. Thus $\pi^{dk}C_k$ would be the unique term of least valuation in~\eqref{eq:precision}, contradicting $\Delta=0$. All residues in~\eqref{eq:Cres} therefore vanish, proving~\eqref{eq:Hequal}.

The products in~\eqref{eq:H} now agree modulo $T^h$. Since $h<p$, their logarithms are defined modulo $T^h$ over $\F_p$. The coefficient of $T^k$ in the difference of these logarithms is
\[
 \lambda_k(d)\sum_{i=0}^{r-1}
 \bigl(b^{-kp^i}-c^{-kp^i}\bigr),
\]
which is \eqref{eq:traceequal}.
\end{proof}

The following two forms of the moment argument will be used to recover the Frobenius orbit.
\begin{lemma}[Power sums]\label{lem:moments}
Let $x_1,\ldots,x_r,y_1,\ldots,y_r$ be nonzero elements of a field of characteristic $p>r$. Their two multisets are equal if either of the following holds:
\begin{enumerate}
\item $\sum_i x_i^j=\sum_i y_i^j$ for $1\le j\le r$;
\item for some integer $u\ge0$, the same equalities hold for $u+1\le j\le u+2r$.
\end{enumerate}
\end{lemma}
\begin{proof}
In the first case, Newton's identities give equality of the elementary symmetric functions, since $1,\ldots,r$ are invertible. Hence $\prod_i(X-x_i)=\prod_i(X-y_i)$, and the multisets coincide.

In the second case, list the distinct elements in the union of the multisets as $z_1,\ldots,z_t$, where $t\le2r$, and let $n_i$ be the difference of their integer multiplicities. The first $t$ of the assumed equations give
\[
 \sum_{i=1}^t n_i z_i^{u+j}=0\qquad(1\le j\le t).
\]
The determinant of this system is
\[
 \left(\prod_{i=1}^t z_i^{u+1}\right)
 \prod_{1\le i<\ell\le t}(z_\ell-z_i),
\]
which is nonzero. Thus $n_i=0$ in the field. Since $|n_i|\le r<p$, every $n_i$ is zero as an integer.
\end{proof}

\begin{theorem}[Distinctness criterion]\label{thm:criterion}
Let $d\mid p-1$ and $h=(p-1)/d$. Suppose either
\begin{enumerate}
\item $r<h$ and $\lambda_j(d)\not\equiv0\pmod p$ for $1\le j\le r$; or
\item there is an integer $u\ge0$ such that $u+2r<h$ and $\lambda_j(d)\not\equiv0\pmod p$ for $u+1\le j\le u+2r$.
\end{enumerate}
Then, for $b,c\in\F_{p^r}^\times$,
\[
 \IK_{d-1}(p^r,b)=\IK_{d-1}(p^r,c)
 \quad\Longleftrightarrow\quad c=b^{p^j}\text{ for some }0\le j<r.
\]
\end{theorem}
\begin{proof}
Under either hypothesis, $p>r$. By Proposition~\ref{prop:traces}, equality of the sums gives equality of the power sums required in the corresponding case of Lemma~\ref{lem:moments}. Therefore the multisets
\[
 \{b^{-1},b^{-p},\ldots,b^{-p^{r-1}}\},\qquad
 \{c^{-1},c^{-p},\ldots,c^{-p^{r-1}}\}
\]
coincide. In particular, $c^{-1}=b^{-p^j}$ for some $0\le j<r$. The converse follows from~\eqref{eq:frob}.
\end{proof}

\section{Galois stabilizers and algebraic degrees}\label{sec:degree}
Once distinctness is known, the degree is determined by the scalar symmetries of the Frobenius orbit. Recall the integer $g_b$ defined in~\eqref{eq:gb}.

\begin{proposition}[Degree formula]\label{prop:degree}
Assume that, for all $b,c\in\F_{p^r}^\times$, equality $\IK_{d-1}(p^r,b)=\IK_{d-1}(p^r,c)$ implies that $b$ and $c$ are Frobenius conjugate. Then, for every $b\in\F_{p^r}^\times$,
\begin{equation}\label{eq:generaldegree}
 \degree\IK_{d-1}(p^r,b)=\frac{p-1}{\gcd(dg_b,p-1)}.
\end{equation}
\end{proposition}
\begin{proof}
By~\eqref{eq:galois} and the distinctness hypothesis, $\sigma_a$ fixes $\IK_{d-1}(p^r,b)$ if and only if $a^{-d}b$ is a Frobenius conjugate of $b$. For $t\in\F_p^\times$, the minimal polynomial of $tb$ is
\[
 t^sP_b(X/t)=X^s+c_1tX^{s-1}+\cdots+c_st^s.
\]
Thus $tb$ is conjugate to $b$ if and only if $t^k=1$ for every $k$ with $c_k\ne0$, equivalently $t^{g_b}=1$. Taking $t=a^{-d}$, we obtain the stabilizer
\[
 \{a\in\F_p^\times:a^{dg_b}=1\},
\]
whose order is $\gcd(dg_b,p-1)$. Its index in $\Gal(\Q(\zeta_p)/\Q)$ is the degree in~\eqref{eq:generaldegree}.
\end{proof}

\section{Proofs of the main theorems}\label{sec:main}
\subsection{Quadratic extensions}
We first record the nonvanishing assertion used in Theorem~\ref{thm:quadratic}.

\begin{lemma}\label{lem:Dnonzero}
For every integer $d\ge2$, the integer $D_d$ defined in~\eqref{eq:D} is nonzero.
\end{lemma}
\begin{proof}
Put
\[
 R_d=\frac{2(d+1)^d}{\binom{2d}{d}^{\!2}}.
\]
For $d=2,\ldots,10$, the values of $D_d$ are, respectively,
\begin{align*}
 D_d={}&-18,\ -272,\ -3650,\ -47952,\ -618478,\\
       &-7584320,\ -79543458,\ -363904400,\ 17740069666.
\end{align*}
Moreover, Bernoulli's inequality gives
\begin{align*}
 \frac{R_{d+1}}{R_d}
 &=\frac{(d+2)(d+1)^2}{4(2d+1)^2}
   \left(1+\frac1{d+1}\right)^d\\
 &\ge\frac{(d+2)(d+1)}{4(2d+1)}>1\qquad(d\ge6).
\end{align*}
Thus $R_d>1$ for every $d\ge10$, since $R_{10}>1$. Together with the displayed values, this proves the assertion.
\end{proof}

\begin{proof}[Proof of Theorem~\ref{thm:quadratic}]
The first two coefficients of \eqref{eq:fseries} are
\[
 a_1(d)=\frac1{d^d(d!)^2},\qquad
 a_2(d)=\frac{(d+1)^d}{d^{2d}((2d)!)^2}.
\]
Consequently,
\begin{equation}\label{eq:lambda12}
 \lambda_1(d)=a_1(d),\qquad
 \lambda_2(d)=a_1(d)^2\frac{D_d}{2\binom{2d}{d}^{\!2}}.
\end{equation}
Under \eqref{eq:quadcond}, both coefficients are nonzero modulo $p$ and $h>2$. The first case of Theorem~\ref{thm:criterion}, with $r=2$, proves \eqref{eq:quaddist}.

To obtain the degree formula, first suppose $b\in\F_p^\times$. Then $g_b=1$ and $\Tr_{\F_{p^2}/\F_p}(b)=2b\ne0$. If $b\notin\F_p$, its minimal polynomial is
\[
 P_b(X)=X^2-\Tr(b)X+b^{p+1}.
\]
Thus $g_b=1$ if $\Tr(b)\ne0$ and $g_b=2$ if $\Tr(b)=0$. Proposition~\ref{prop:degree}, together with $d\mid p-1$, proves \eqref{eq:quaddegree}.
\end{proof}

\subsection{Two-dimensional sums}
\begin{proof}[Proof of Theorem~\ref{thm:twodim}]
For $d=3$, the first five logarithmic coefficients are
\begin{equation}\label{eq:lambdas3}
\begin{aligned}
 \lambda_1(3)&=\frac1{972},&
 \lambda_2(3)&=-\frac{17}{47239200},\\
 \lambda_3(3)&=\frac{163}{826497043200},&
 \lambda_4(3)&=-\frac{200407}{1620099504080640000},\\
 \lambda_5(3)&=\frac{457217}{5511578512882337280000},&&
\end{aligned}
\end{equation}
where $457217=23\cdot103\cdot193$. They are obtained from the recurrence
\begin{equation}\label{eq:recurrence}
 \lambda_j(d)=a_j(d)-\frac1j\sum_{k=1}^{j-1}k\lambda_k(d)a_{j-k}(d),
\end{equation}
which follows by comparing coefficients in $f_d'=f_d(\log f_d)'$.

The hypothesis $p>3r+1$ ensures that $r<h=(p-1)/3$ and that $\lambda_1(3),\ldots,\lambda_r(3)$ are $p$-integral. Their numerators, together with $p\equiv1\pmod3$, leave only the following primes to consider separately:
\[
 103,\quad163,\quad193,\quad200407.
\]
For every other prime satisfying~\eqref{eq:twocond}, these coefficients are nonzero modulo $p$, so the first case of Theorem~\ref{thm:criterion} applies.

For each of the four remaining primes, recurrence~\eqref{eq:recurrence} gives the following block of ten nonzero residues. The entries in each row are listed in increasing order of the index $j$.
\begin{center}
\small
\setlength{\tabcolsep}{5pt}
\begin{tabular}{rcl}
\toprule
$p$ & Indices $j$ & Residues of $\lambda_j(3)$ modulo $p$\\
\midrule
$163$ & $4$--$13$ & $161,26,157,140,54,1,148,56,109,121$\\
$200407$ & $5$--$14$ & $180025,96682,14235,112480,111584,$\\
& & $30660,133582,65716,7932,77321$\\
$103$ & $6$--$15$ & $60,42,22,75,42,73,1,84,77,62$\\
$193$ & $6$--$15$ & $165,151,105,132,99,34,3,137,100,118$\\
\bottomrule
\end{tabular}
\end{center}
In each row the last index is less than $(p-1)/3$. Taking the first $2r$ entries of that row therefore satisfies the second hypothesis of Theorem~\ref{thm:criterion} for every $r\le5$. This proves~\eqref{eq:twodist} for the remaining primes. Finally, Proposition~\ref{prop:degree} with $d=3$ gives~\eqref{eq:twodegree}.
\end{proof}

\begin{remark}
The condition $p\nmid D_d$ is sufficient, but is not asserted to be necessary: it ensures that $\lambda_2(d)$ is nonzero modulo $p$. For $d=3$, the identity $D_3=-272=-16\cdot17$ shows that this condition excludes no odd prime $p\equiv1\pmod3$. Thus Theorems~\ref{thm:quadratic} and~\ref{thm:twodim} have the same hypotheses when $d=3$ and $r=2$.
\end{remark}

\section{Consequences}\label{sec:consequences}
Counting Frobenius orbits gives the number of distinct values in each family.

\begin{corollary}\label{cor:count}
Under the hypotheses of Theorem~\ref{thm:quadratic},
\[
 \#\{\IK_{d-1}(p^2,b):b\in\F_{p^2}^\times\}
 =\frac{p^2+p-2}{2}.
\]
Under the hypotheses of Theorem~\ref{thm:twodim},
\begin{equation}\label{eq:burnside}
 \#\{\IK_2(p^r,b):b\in\F_{p^r}^\times\}
 =\frac1r\sum_{j=0}^{r-1}\bigl(p^{\gcd(r,j)}-1\bigr),
\end{equation}
where $\gcd(r,0)=r$.
\end{corollary}
\begin{proof}
The distinct values are in bijection with the Frobenius orbits on $\F_{p^r}^\times$. The fixed points of the $j$th power of Frobenius are $\F_{p^{\gcd(r,j)}}^\times$. Burnside's lemma gives \eqref{eq:burnside}, and its specialization to $r=2$ gives the first assertion.
\end{proof}

For cubic extensions, the degree formula can be expressed directly in terms of the parameter.
\begin{corollary}\label{cor:cubic}
If $p\equiv1\pmod3$ and $p\ge13$, then, for $b\in\F_{p^3}^\times$,
\[
 \degree\IK_2(p^3,b)=
 \begin{cases}
 (p-1)/\gcd(9,p-1),&b\notin\F_p\text{ and }b^3\in\F_p,\\[2pt]
 (p-1)/3,&\text{otherwise}.
 \end{cases}
\]
\end{corollary}
\begin{proof}
The degree $s$ of $b$ over $\F_p$ is either $1$ or $3$. If $s=1$, then $g_b=1$. If $s=3$, the nonzero constant coefficient forces $g_b\mid3$. We have $g_b=3$ precisely when $P_b(X)=X^3-v$ for some $v\in\F_p^\times$, equivalently when $b^3\in\F_p$. Apply Theorem~\ref{thm:twodim}.
\end{proof}

\begin{remark}
For a parameter satisfying $b\notin\F_p$ and $b^3\in\F_p$, the degree in Corollary~\ref{cor:cubic} is smaller than $(p-1)/3$ precisely when $9\mid p-1$. In general, the stabilizer in Proposition~\ref{prop:degree} describes the degree reduction through the scalar symmetries of the minimal polynomial.
\end{remark}


\begin{thebibliography}{9}
\bibitem{GK}
B. H. Gross and N. Koblitz, \emph{Gauss sums and the $p$-adic $\Gamma$-function}, Ann. of Math. (2) \textbf{109} (1979), no.~3, 569--581. \href{https://doi.org/10.2307/1971226}{doi:10.2307/1971226}.
\bibitem{KRV}
K. Kononen, M. Rinta-aho and K. V\"a\"an\"anen, \emph{On the degree of a Kloosterman sum as an algebraic integer}, preprint, 2011, \href{https://arxiv.org/abs/1107.0188}{arXiv:1107.0188}.
\bibitem{Lin25}
X. Lin, \emph{On algebraic degrees of certain exponential sums over finite fields}, Forum Math. \textbf{37} (2025), no.~5, 1433--1438. \href{https://doi.org/10.1515/forum-2024-0365}{doi:10.1515/forum-2024-0365}.
\bibitem{Lin26}
X. Lin, \emph{Distinctness of the generalized $n$-dimensional Kloosterman sums}, preprint, 2026, \href{https://arxiv.org/abs/2607.04177}{arXiv:2607.04177}.
\bibitem{LWfields}
X. Lin and D. Wan, \emph{On inverted Kloosterman sums over finite fields}, Math. Z. \textbf{306} (2024), article~60. \href{https://doi.org/10.1007/s00209-024-03457-0}{doi:10.1007/s00209-024-03457-0}.
\bibitem{LWdegree}
X. Lin and D. Wan, \emph{On algebraic degrees of inverted Kloosterman sums}, Finite Fields Appl. \textbf{99} (2024), article~102477. \href{https://doi.org/10.1016/j.ffa.2024.102477}{doi:10.1016/j.ffa.2024.102477}.
\bibitem{Wan}
D. Wan, \emph{Minimal polynomials and distinctness of Kloosterman sums}, Finite Fields Appl. \textbf{1} (1995), no.~2, 189--203. \href{https://doi.org/10.1006/ffta.1995.1015}{doi:10.1006/ffta.1995.1015}.
\end{thebibliography}
\end{document}